\documentclass[12pt,reqno]{amsart}
\usepackage{amsmath}
\usepackage{amsthm}
\usepackage{amssymb}
\usepackage{amsfonts}
\usepackage{latexsym}
\usepackage{hyperref}
\usepackage{cite}
\usepackage{xcolor}
\usepackage{fullpage}

\usepackage{todonotes}
\usepackage{comment}
\usepackage{multirow}

\theoremstyle{plain}
\newtheorem{theorem}{Theorem}[section]
\newtheorem{corollary}[theorem]{Corollary}
\newtheorem{lemma}[theorem]{Lemma}

\theoremstyle{definition}

\newtheorem*{remark}{Remark}
\newtheorem*{remarks}{Remarks}

\renewcommand{\L}{\mathbb{L}}

\newcommand{\Q}{\mathbb{Q}}
\newcommand{\D}{\mathbf{D}}

\newcommand{\K}{\mathbb{K}}
\newcommand{\Z}{\mathbb{Z}}

\newcommand{\cAA}{\mathfrak{A}}

\newcommand{\RR}{\mathcal{R}}

\newcommand{\cPP}{\mathfrak{P}}

\renewcommand{\pmod}[1]{\,\,(\operatorname{mod}#1)}

\let\oldenumerate=\enumerate
	\def\enumerate{
	\oldenumerate
	\setlength{\itemsep}{5pt}
	}
\let\olditemize=\itemize
	\def\itemize{
	\olditemize
	\setlength{\itemsep}{5pt}
	}

\allowdisplaybreaks

\begin{document}

\title[Primes representable by polynomials]{Explicit upper bounds for the number of primes simultaneously representable by any set of irreducible polynomials}

\author[M.~Bordignon]{Matteo Bordignon}
\address{Charles University, Faculty of Mathematics and Physics, Department of Algebra, Sokolovsk 83, 186 00 Praha 8, Czech Republic}
\email{matteobordignon91@gmail.com}

\author[E.~S.~Lee]{Ethan Simpson Lee}
\address{University of Bristol, School of Mathematics, Fry Building, Woodland Road, Bristol, BS8 1UG} 
\email{ethan.lee@bristol.ac.uk}
\urladdr{\url{https://sites.google.com/view/ethansleemath/home}}




\maketitle


\begin{abstract}
    Using an explicit version of Selberg's upper sieve, we obtain explicit upper bounds for the number of $n\leq x$ such that a non-empty set of irreducible polynomials $F_i(n)$ with integer coefficients are simultaneously prime; this set can contain as many polynomials as desired. To demonstrate, we present computations for some irreducible polynomials and obtain an explicit upper bound for the number of Sophie Germain primes up to $x$, which have practical applications in cryptography.
\end{abstract}

\section{Introduction}

Let $F_1, \dots, F_g\in\Z[X]$ be irreducible polynomials with positive leading coefficients, suppose that $F = F_1\cdots F_g$, and let $\rho_F(p)$ denote the number of solutions to the congruence $F(n)\equiv 0\mod{p}$, such that $\rho_F(p) < p$ for all primes $p$. This condition is equivalent to observing that $H$ has no fixed prime divisors. Throughout, $p$ will always denote a rational prime number and the Generalised Riemann Hypothesis is shortened to the GRH.

A central problem in number theory is to understand the distribution of the prime numbers in a set $\mathcal{S} \subseteq \Z$. For example, Maynard considered the set $\mathcal{S}$ of elements defined by linear equations in his thesis \cite{MaynardThesis}. More generally, suppose that $\pi_F(y)$ is the number of integers $n\leq y$ such that $F_i(n)$ is prime for each $1\leq  i\leq  g$; this considers the set $\mathcal{S}$ of integers such that $F_i(n)$ are simultaneously prime. The Bateman--Horn conjecture \cite{BatemanHornA, BatemanHornB} is a generalisation of six conjectures of Hardy and Littlewood \cite{HardyLittlewood23} which suggests that
\begin{equation}\label{eq:BHC}
    \pi_F(y)
    \sim \frac{1}{\prod_{i=1}^g\deg{F_i}} \prod_{p}\left(1-\frac{1}{p}\right)^{-g}\left(1 - \frac{\rho_F(p)}{p}\right) \int_2^y\frac{dt}{\log^g{t}}.
\end{equation}
Note that the prime number theorem for primes in arithmetic progressions is a proven special case of the Bateman--Horn conjecture and a comprehensive account on the Bateman--Horn conjecture has been given by Aletheia \textit{et al.} in \cite{AZFG}. 

The asymptotic behaviour in \eqref{eq:BHC} is difficult to prove, but Halberstam and Richert \cite[Thm.~5.3]{HalberstamRichert} have used the upper Selberg sieve to establish
\begin{equation}\label{eq:RH}
    \pi_F(y) \leq  2^g g! \prod_{p}\left(\frac{p - \rho_F(p)}{p-1}\right)\left(1-\frac{1}{p}\right)^{1-g} \frac{y}{\log^g{y}}\left(1 + O_F\left(\frac{\log\log{3y}}{\log{y}}\right)\right),
\end{equation}
in which the implied constant is independent of $x$ and might depend on $F$. The primary objective of this paper is to prove Theorem \ref{theo:general}, which is an effective version of \eqref{eq:RH}, by making every step in Halberstam and Richert's exposition completely explicit. To this end, we needed to apply explicit ingredients that have been established by the second author in his thesis \cite{LeeThesis} into a general, but completely explicit, version of Selberg's upper sieve, which we prove in Theorem \ref{theo:sieve}.

Once Theorem \ref{theo:general} has been proved, we can establish some applications of it. A particularly interesting special case of the explicit version of \eqref{eq:RH} is when $g=1$, because this would reveal an upper bound for the number of primes representable by an irreducible polynomial. To this end, we prove the following result.

\begin{corollary}\label{thm:Main}
Suppose that 
\begin{equation*}
    F_0(k) = k^2 + 3,\quad
    F_1(k) = k^3 - 5,\quad
    F_2(k) = k^5 + 3,\quad
    \text{and}\quad
    F_3(k) = 2 k^6 + 3.
\end{equation*}
For all $x\geq X_i$, there exists $\tau_i > 0$ such that
\begin{equation*}
    \pi_{F_i}(x) \leq 2 \left( 1+\frac{\tau_i \log\log{x}}{\log{x}}\right)  \prod_p\left( 1-\frac{\rho_{F_i}(p)}{p}\right)\left( 1-\frac{1}{p}\right)^{-1} \frac{x}{\log{x}} .
\end{equation*}
Admissible values for $X_i$ and $\tau_i$ are presented in Table \ref{tab:Main}, including computations that are conditional on the GRH. 
\end{corollary}

\begin{table}[]
    \centering
    \begin{tabular}{c|cc|cc}
    \multirow{2}{*}{$i$} & \multicolumn{2}{c|}{$\log{X_i}$} & \multicolumn{2}{c}{$\log{\tau_i}$} \\
    & unconditionally & under GRH  & unconditionally & under GRH  \\
    \hline
    $0$ & $1.5\cdot 10^{38}$ & $5.5\cdot 10^{7}$ & $2.41\cdot 10^{25}$ & $1.29\cdot 10^{5}$ \\
    $1$ & $1.8\cdot 10^{50}$ & $5.7\cdot 10^{8}$ & $3.01\cdot 10^{33}$ & $6.73\cdot 10^{5}$ \\
    $2$ & $6.1\cdot 10^{75}$ & $6.5\cdot 10^{9}$ & $3.70\cdot 10^{50}$ & $3.88\cdot 10^{6}$ \\
    $3$ & $2.8\cdot 10^{109}$ & $9.3\cdot 10^{10}$ & $1.07\cdot 10^{73}$ & $2.41\cdot 10^{7}$ \\
    \end{tabular}
    \caption{Admissible values for $X_i$ and $\tau_i$ in Corollary \ref{thm:Main}, which correspond to the irreducible polynomials $F_i$.}
    \label{tab:Main}
\end{table}

We used \texttt{Python} to determine the values in Table \ref{tab:Main}. The code we used to determine these values is now available on \href{https://github.com/EthanSLee/PrimesRepresentableByPolynomials/blob/main/IrreduciblePolynomialTool.py}{\texttt{GitHub}}, so that the motivated reader can run our computations and determine admissible values for $X_i$ and $\tau_i$ under \textit{any} choice of irreducible $F\in\Z[X]$. Now, the computations in Table \ref{tab:Main} clearly indicate that irreducible polynomials of larger degree will require shorter ranges for $x$ and larger constants $\tau_i$. 

Next, the results of this paper enable one to obtain explicit results for \textit{any} system of $g\geq 1$ irreducible polynomials. The only consideration to keep in mind, is that we are limited by technology, because some computations will be too computationally expensive. To demonstrate how our results may be used to obtain explicit bounds for a more general system of polynomials, we consider another interesting special case of \eqref{eq:RH}. That is, if $g=2$, then we can obtain an upper bound for the number of Sophie Germain primes, where a prime $p$ is \textit{Sophie Germain} if $2p+1$ is also prime; these are also known as \textit{safe} primes in cryptography. Suppose that $H(k) = k(2k + 1)$, so that
\begin{equation*}
    \pi_H(x) = \#\{p \leq x : p~\text{and}~ 2p+1 ~\text{are prime}\} 
\end{equation*}
is the number of Sophie Germain primes up to $x$. There is a heuristic conjecture that is a special case of Dickson's conjecture (see \cite[Conj.~5.26]{Shoup}), which states
\begin{equation*}
    \pi_H(x) \sim 2 \prod_{p>2}\left(1-\frac{1}{(p-1)^2} \right)\frac{x}{\log^2{x}}.
\end{equation*}
Note that this observation is also a special case of the wider Bateman--Horn conjecture \eqref{eq:BHC}. In the end, we prove the following unconditional result, which is close to the expected bound.

\begin{corollary}\label{thm:Main1}
If $\log{x} \geq 1.3\cdot 10^{6}$, then
\begin{equation*}
    \pi_H(x) \leq 16 \left(1 + \frac{e^{9.04\cdot 10^{3}} \log\log{x}}{\log{x}}\right) \prod_{p>2} \left( 1-\frac{1}{(p-1)^2}\right) \frac{x}{\log^2{x}} .
\end{equation*}
\end{corollary} 

\begin{remarks}
The products over primes in our results will converge and the function $\rho_{F}(p)$ can be analysed further on a case-by-case basis. In particular, Wrench computed that the product in Corollary \ref{thm:Main1} is equal to $0.6601618158\ldots$ in \cite{Wrench}; his computations are precise up to $42$ decimal places. Moreover, we can note that 
\begin{equation*}
    \rho_{F_0}(p) =
    \begin{cases}
        1 & \text{if }p\in\{2,3\},\\
        1+\left(\frac{-3}{p}\right) &\text{if }p>3,
    \end{cases}
\end{equation*}
where $\left(\tfrac{\cdot}{\cdot}\right)$ is the Legendre symbol. Finally, write $f = O^*(g)$ to mean $|f| \leq g$ throughout.
\end{remarks}

\subsection*{Structure}
In Section \ref{S:Selb} we prove a completely explicit version of Selberg's upper sieve, by following the work of Halberstam and Richert in \cite{HalberstamRichert}. In Section \ref{S:app}, we make choices in this sieve and apply the explicit ingredients from \cite{LeeThesis} to yield Theorems \ref{theo:general}-\ref{theo:generalp}. Finally, in Section \ref{S:Results}, we apply these results and obtain the computations presented in Corollaries \ref{thm:Main} and \ref{thm:Main1}.

\subsection*{Acknowledgements}
We would like to thank Olivier Ramar\'{e} for suggesting this project. Lee would like to thank the University of New South Wales and the Heilbronn Institute for Mathematical Research for their support. Bordignon was supported in his research by OP RDE
project No.CZ.02.2.69/0.0/0.0/18\_053/0016976 International mobility of research, technical and administrative staff at the Charles University.


\section{Selberg's upper sieve}\label{S:Selb}

In this section, we introduce Selberg's upper sieve from \cite{HalberstamRichert} and make every aspect of this sieve explicit. That is, all implied constants will be described using constants.

\subsection{Notation and initial result}

Suppose $\cAA = (a_n)$ is a finite sequence of non-negative real numbers $a_n$, $\cPP$ is a set of primes, and $\cPP^C$ is the complement of $\cPP$ in the set of all primes. The objective of sieve theory is to obtain bounds for the sifted set
\begin{equation*}
    S(\cAA, \cPP, z) = \#\{a\in\cAA : (a,P(z))=1\}
    \quad\text{such that}\quad
    P(z) = \prod_{\substack{p\in\cPP\\p\leq  z}} p .
\end{equation*}
The set $\cAA_d$ contains all the elements of $\cAA$ which are divisible by $d$. For square-free $d$, let $\omega$ be a multiplicative function which approximates $\#\cAA_d$ by the relationship
\begin{equation*}
    \#\cAA_d = \frac{\omega(d)}{d}X + \RR_d,
\end{equation*}
in which $X > 0$ and $\RR_d$ is a remainder term. Moreover, let
\begin{equation*}
    W(z) = \prod_{p<z}\left(1-\frac{\omega(p)}{p}\right),\quad
    g(d) = \frac{\omega(d)}{d}\prod_{p|d} \left(1-\frac{\omega(p)}{p}\right)^{-1},
    \quad\text{and}\quad
    G(z) = \sum_{d < z}\mu^2(d)g(d).
\end{equation*}
Finally, the following conditions for square-free $d$ such that $(d, \cPP^C)=1$ and prime numbers $p$ will be important. For constants $\kappa \geq 1$, $A_1 > 1$, and $A_2 > 0$, these are:
\begin{align}
    \#\RR_{d} &\leq  \omega(d),\tag{$R$}\label{eq:R}\\
    0 &\leq  \frac{\omega(p)}{p} \leq  1 - \frac{1}{A_1}, \tag{$\Omega_1$}\label{eq:omega1}\\
    -L &\leq  \sum_{w\leq  p<z}\frac{\omega(p)\log{p}}{p} - \kappa \log{\frac{z}{w}} \leq  A_2.\tag{$\Omega_2(\kappa)$}\label{eq:omega2}
\end{align}
The condition \eqref{eq:omega1} ensures that $g(d)$ is well-defined and satisfies
\begin{equation}\label{eq:gp_bounds}
    \frac{\omega(p)}{p} \leq  g(p) \leq  A_1 \frac{\omega(p)}{p}.
\end{equation}

Suppose that $\lambda_1 = 1$ and $\lambda_d$ for $d\geq 2$ is a sequence of arbitrary real numbers. Selberg's upper sieve is derived from the inequality
\begin{equation*}
    S(\cAA,\cPP,z) \leq  \sum_{a\in\cAA} \bigg(\sum_{d|(a,P(z))}\lambda_d\bigg)^2.
\end{equation*}
Following Halberstam and Richert \cite[Thm.~3.1, Thm.~3.2]{HalberstamRichert}, one may obtain the following theorem, assuming the preceding knowledge and notation.

\begin{theorem}[Selberg's upper sieve]\label{thm:SelbergUpperSieve}
Suppose that conditions \eqref{eq:R} and \eqref{eq:omega1} hold, then 
\begin{equation*}
    S(\cAA,\cPP,z) \leq  \frac{X}{G(z)} + \frac{z^2}{W^3(z)}.
\end{equation*}
\end{theorem}

\subsection{Technical preparations}

Here, we establish some technical results, which will be needed to find explicit bounds for $G(z)$ later. First, we require \cite[Thm.~7]{Rosser}: 
\begin{equation}\label{eq:4}
    V(z) = \prod_{p<z}\left(1-\frac{1}{p}\right)=\frac{e^{-\gamma}}{\log z}\left( 1+O^*\left( \frac{1}{\log^2 z}\right)\right). 
\end{equation}
It also follows from \eqref{eq:omega2} that
\begin{equation}\label{eq:5}
    -\frac{L}{\log w}
    \leq  \sum_{w\leq  p <z}\frac{\omega(p)}{p} - \kappa\log\frac{\log z}{\log w}
    \leq  \frac{A_2}{\log{w}};
\end{equation}
see the proof of \cite[(5.2.1)]{HalberstamRichert} for more information. Next, the following result is an explicit version of \cite[(5.2.2)]{HalberstamRichert}. 

\begin{lemma}\label{lemm:exp1}
If $2\leq w \leq z$, then we have
\begin{align*}
    \sum_{w \leq p \leq z}\frac{g(p)}{p^s}- \kappa\sum_{w \leq p \leq z}\frac{1}{p^{s+1}} = 
     O^*\left(\max\left\{\frac{A_2}{\log w}+\frac{A_1A_2}{\log w}\left(\kappa+ \frac{A_2}{\log w}\right), \frac{L}{\log w}\right\}+\frac{3\kappa}{2\log^2 w}\right),
\end{align*}
uniformly in $s\geq 0$.
\end{lemma}

\begin{proof}
Begin by using \cite[(2.3.11)]{HalberstamRichert}, which states
\begin{align*}
    \sum_{w \leq p <z}g(p)\leq \kappa \log\frac{\log z}{\log w}+\frac{A_2}{\log w}+\frac{A_1A_2}{\log w}\left(\kappa + \frac{A_2}{\log w}\right),
\end{align*} 
and \cite[Thm.~5]{Rosser} to obtain
\begin{equation}\label{eq:upper}
    \sum_{w \leq p \leq z}g(p)- \kappa\sum_{w \leq p \leq z}\frac{1}{p} \leq \frac{A_2}{\log w}+\frac{A_1A_2}{\log w}\left(\kappa + \frac{A_2}{\log w}\right)+\frac{3\kappa}{2\log^2{w}}.
\end{equation}
So, combine \eqref{eq:5} with the lower bound in \eqref{eq:gp_bounds} and \cite[Thm.~5]{Rosser} to obtain
\begin{equation}\label{eq:lower}
    -\left(\frac{L}{\log w} + \frac{3\kappa}{2\log^2 w}\right)
    \leq  \sum_{w\leq p <z}\frac{\omega(p)}{p} - \kappa \sum_{w \leq p \leq z}\frac{1}{p}
    \leq  \sum_{w \leq p \leq z}g(p)- \kappa\sum_{w \leq p \leq z}\frac{1}{p}.
\end{equation}
The result now follows from \eqref{eq:upper} and \eqref{eq:lower} by partial summation.
\end{proof}

The next lemma is an intermediate observation that is required to prove another lemma.

\begin{lemma}
\label{lemm:exp2}
If $2\leq w < z$ and $k\geq 2$, then we have
\begin{equation}\label{eq:2}
    \sum_{w\leq p <z}g^k(p)\leq \frac{A_1^k A_2^{k-1}}{(\log{w})^{k-1}} \left(\kappa +\frac{A_2}{\log w}\right).
\end{equation}
\end{lemma}

\begin{proof}
Use the upper bound in \eqref{eq:gp_bounds} to obtain 
\begin{align*}
    \sum_{w\leq p <z}g^k(p)\leq A_1^k \sum_{w\leq p <z}\left( \frac{\omega(p)}{p}\right)^k.
\end{align*}
Next, use \cite[(2.3.8)]{HalberstamRichert}, which tells us that $\frac{\omega(p)\log p}{p}\leq A_2$, to obtain 
\begin{align*}
    \sum_{w\leq p <z}g^k(p)
    \leq  A_1^k A_2^{k-1} \sum_{w\leq p <z} \frac{\omega(p)}{p(\log p)^{k-1}}
    \leq  \frac{A_1^k A_2^{k-1}}{(\log{w})^{k-2}} \sum_{w\leq p <z} \frac{\omega(p)}{p\log p}.
\end{align*}
The result follows using \cite[(2.3.7)]{HalberstamRichert}, which states $\sum_{w\leq p <z} \frac{\omega(p)}{p\log p}\leq  \frac{1}{\log w}\left( \kappa +\frac{A_2}{\log w}\right)$.
\end{proof}

Using the preceding lemma, we can now prove an explicit version of \cite[Lem.~5.5.3]{HalberstamRichert}, which bounds the product $W(z)$.

\begin{lemma}\label{lem:W}
If $z > \exp(A_1A_2)$, then
\begin{align}
\label{eq:W}
    W(z)
    = \prod_p\left( 1-\frac{\omega(p)}{p}\right)\left( 1-\frac{1}{p}\right)^{-\kappa}\frac{e^{-\kappa\gamma}}{\log^\kappa{z}}\left(1 + O^*\left(\frac{\widehat{m_0}(z)}{\log z}\right) \right),
\end{align}
where $\widehat{m_0}(z) = \log{z} \left(\left(1 + \frac{1}{\log^2{z}}\right)^{\kappa} (1+ m_0(z) \exp(m_0(z))) - 1\right)$ and, for any $k_0\geq 2$,
\begin{align*}
    m_0(w) 
    &=\max\left\{\frac{A_2}{\log w}+\frac{A_1A_2}{\log w}\left(\kappa+ \frac{A_2}{\log w}\right), \frac{L}{\log w}\right\} + \frac{3\kappa}{2\log^2 w} + \kappa \log\frac{w}{w-1}\\
    &+ \frac{A_1^2 A_2}{\log{w}}\left(\kappa +\frac{A_2}{\log w}\right)\left(\sum_{k=2}^{k_0} \frac{A_1^{k-2} A_2^{k-2}}{ k(\log{w})^{k-2}}+\frac{A_1^{k_0-1} A_2^{k_0-1}}{ (\log{w})^{k_0-1}}\left((k_0+1)\left(1-\frac{A_1 A_2}{ (\log{w})}\right)\right)^{-1} \right).
\end{align*}
\end{lemma}

\begin{proof}
We first aim to show that if $\exp(A_1A_2) < w \leq z$, then
\begin{align}
\label{eq:prod}
    \prod_{w\leq p <z}&\left(1+\frac{g(p)}{p^s} \right)\left(1-\frac{1}{p^{s+1}} \right)^{\kappa}
    = 1 + O^*\left(m_0(w) \exp(m_0(w))\right) .
\end{align}
To prove this, let $s\geq 0$ and observe that $0\leq  -y-\log(1-y) \leq  \frac{y^2}{1-y}$ for $0\leq  y <1$, so
\begin{equation*}
    \log\left(1-\frac{1}{p^{s+1}} \right)
    = - \frac{1}{p^{s+1}} + O^*\left(\frac{1}{p(p-1)}\right).
\end{equation*}
Next, use the Taylor series expansion for $\log(1+x)$ for $|x| < 1$ and Lemma \ref{lemm:exp2} to see that
\begin{align*}
    \sum_{w\leq p <z} \log\left(1+\frac{g(p)}{p^s} \right)
    &= \sum_{w\leq p <z} \left(\frac{g(p)}{p^{s}} + \sum_{k=2}^\infty (-1)^{k+1}\frac{g^k(p)}{k p^{s k}}\right) \\
    &= \sum_{w\leq p <z} \frac{g(p)}{p^{s}} + O^*\left(\sum_{w\leq p <z} \sum_{k=2}^\infty \frac{g^k(p)}{k}\right) \\
    &= \sum_{w\leq p <z} \frac{g(p)}{p^{s}} + O^*\left(\sum_{k=2}^\infty \frac{1}{k} \sum_{w\leq p <z} g^k(p)\right) \\
    &= \sum_{w\leq p <z} \frac{g(p)}{p^{s}} + O^*\left(\left(\kappa + \frac{A_2}{\log w}\right) \sum_{k=2}^\infty \frac{A_1^k A_2^{k-1}}{k (\log{w})^{k-1}} \right) \\
    &= \sum_{w\leq p <z} \frac{g(p)}{p^{s}} + O^*\left(\frac{A_1^2 A_2}{\log{w}}\left(\kappa +\frac{A_2}{\log w}\right) \sum_{k=2}^\infty \frac{A_1^{k-2} A_2^{k-2}}{k (\log{w})^{k-2}} \right).
\end{align*}
It is interesting to note that there should be room for improvement going from the first to the second line of the above equation, but this would be too technical to implement for a relatively small payout. 
Now, the last sum converges when $|A_1A_2/\log{w}| < 1$, which only occurs when $w > e^{A_1A_2}$. Use these last observations and Lemma \ref{lemm:exp1} to see that if $e^{A_1A_2} < w \leq z$, taken any $k_0\geq 2$, then
\begin{align*}
    &\log{\prod_{w\leq p <z}\left(1+\frac{g(p)}{p^s} \right)\left(1-\frac{1}{p^{s+1}} \right)^{\kappa}} \\ 
    &=\sum_{w\leq p <z} \log\left(1+\frac{g(p)}{p^s} \right) + \kappa\log\left(1-\frac{1}{p^{s+1}} \right) \\ 
    &= \sum_{w\leq p <z} \left(\frac{g(p)}{p^{s}} - \frac{\kappa}{p^{s+1}}\right) + O^*\left(\frac{A_1^2 A_2}{\log{w}}\left(\kappa +\frac{A_2}{\log w}\right) \sum_{k=2}^\infty \frac{A_1^{k-2} A_2^{k-2}}{k (\log{w})^{k-2}} + \sum_{w\leq p <z} \frac{\kappa}{p(p-1)} \right) 
    \\ &=\sum_{w\leq p <z} \left(\frac{g(p)}{p^{s}} - \frac{\kappa}{p^{s+1}}\right) + O^*\left(\frac{A_1^2 A_2}{\log{w}}\left(\kappa +\frac{A_2}{\log w}\right)\cdot\right. \\ 
    & \quad\quad\quad \left.\left(\sum_{k=2}^{k_0} \frac{A_1^{k-2} A_2^{k-2}}{ k(\log{w})^{k-2}}+\frac{A_1^{k_0-1} A_2^{k_0-1}}{k (\log{w})^{k_0-1}}\left((k_0+1)\left(1-\frac{A_1 A_2}{ \log{w}}\right)\right)^{-1} \right) + \int_{w}^\infty \frac{\kappa\,dt}{t(t-1)} \right) \\ 
    &= O^*\left(m_0(w)\right) .
\end{align*}
Now, apply the inequality $|e^t - 1| \leq  |t|e^{|t|}$, which holds for all $t\in\mathbb{R}$, to obtain \eqref{eq:prod}. Note that it was convenient to split the infinite sum at $k_0$ to ensure that the constants in our results can be effectively computed later. 
Finally, for all $w>e^{A_1A_2}$, let $s=0$ and $z\to\infty$ in \eqref{eq:prod} to obtain
\begin{align*}
    \prod_{p\geq  w}\left(1-\frac{\omega(p)}{p}\right)^{-1}\left( 1-\frac{1}{p}\right)^{\kappa} = 1 + O^*\left(m_0(w) \exp(m_0(w))\right).
\end{align*}
This means that the infinite product in \eqref{eq:W} converges and for $z > e^{A_1A_2}$ we can write
\begin{align*}
    W(z)&=\prod_{p<z}\left( 1-\frac{\omega(p)}{p}\right)\prod_{p\geq  z}\left( 1-\frac{\omega(p)}{p}\right)\left( 1-\frac{1}{p}\right)^{-\kappa}\left(1 + O^*\left(m_0(z) \exp(m_0(z))\right) \right)
    \\&= \prod_{p}\left( 1-\frac{\omega(p)}{p}\right)\left( 1-\frac{1}{p}\right)^{-\kappa} V^{\kappa}(z)\left(1 + O^*\left(m_0(z) \exp(m_0(z))\right) \right).
\end{align*}
To complete the result, appeal to \eqref{eq:4}. 
\end{proof}

Our final intermediate lemma follows.

\begin{lemma}\label{lemm:exp3}
For any $x, z, d >0$, we have
\begin{equation}\label{eq:3}
    \sum_{\substack{\sqrt{x/d}\leq p \leq \min(x/d,z)\\p\nmid d}} \frac{g(p)\omega(p) \log{p}}{p} \leq  m_1(x,d),
\end{equation}
in which $m_1(x,d) = A_2 \left(\kappa\log 2+\frac{A_2}{\log \sqrt{x/d}}+\frac{A_1A_2}{\log \sqrt{x/d}}\left(\kappa + \frac{A_2}{\log \sqrt{x/d}}\right)\right)$.
\end{lemma}

\begin{proof}
Using \cite[(2.3.8)]{HalberstamRichert}, we have
\begin{align*}
    \sum_{\substack{\sqrt{x/d}\leq p \leq \min(x/d,z)\\p\nmid d}} \frac{g(p)\omega(p) \log{p}}{p}\leq A_2 \sum_{\sqrt{x/d} \leq p \leq \min(x/d,z)} g(p).
\end{align*}
The result now follows using \cite[(2.3.11)]{HalberstamRichert} which gives
\begin{equation*}
    \sum_{w \leq p <z}g(p) \leq  \kappa \log\frac{\log z}{\log w}+\frac{A_2}{\log w}+\frac{A_1A_2}{\log w}\left(\kappa + \frac{A_2}{\log w}\right). \qedhere
\end{equation*}
\end{proof}

\subsection{Bounds for $G(z)$}

In light of Theorem \ref{thm:SelbergUpperSieve}, we need to find explicit bounds for
\begin{equation*}
    \frac{1}{G(z)} = \left(\sum_{d < z} \frac{\mu^2(d) \omega(d)}{d} \prod_{p|d} \left(1-\frac{\omega(p)}{p}\right)^{-1}\right)^{-1}
    \quad\text{and}\quad
    \frac{1}{W(z)} = \prod_{p<z}\left(1-\frac{\omega(p)}{p}\right)^{-1}.
\end{equation*}
First, the following result provides upper bounds for $W(z)$ and $1/W(z)$.

\begin{theorem}\label{theo:W}
Let $m_2 = \frac{1}{\log^{\kappa}{2}} \exp\left(\frac{A_2}{\log 2} \left(1 + A_1 \kappa +\frac{A_1A_2}{\log 2} \right)\right)$. If $z > 2$, then
\begin{align*}
    W(z)\leq  \frac{\exp\left(\kappa\log\log{2} + \frac{L}{\log{2}}\right)}{\log^{\kappa}{z}}
    \quad\text{and}\quad
    \frac{1}{W(z)}\leq m_2 \log^{\kappa}{z} .
\end{align*}
\end{theorem}

\begin{proof}
The second result follows from \cite[(2.3.12)]{HalberstamRichert}. 
The first result follows from the definition of $W(z)$, \eqref{eq:5}, and the observation $-\frac{y^2}{1-y} \leq  y+\log(1-y) \leq  0$ for $0\leq  y <1$. That is, take $y=\omega(p)/p$ to see
\begin{equation*}
    W(z) 
    \leq \exp\left(-\frac{\omega(p)}{p}\right)
    \leq \exp\left(-\kappa\log\frac{\log{z}}{\log{2}} + \frac{L}{\log{2}}\right). \qedhere
\end{equation*}
\end{proof}

In the next result, we find an upper bound for $1/G(z)$.

\begin{theorem}\label{theo:G}
For $z> 2$ and $\lambda > 0$, we have $G(z)^{-1} \leq  W(z) m_3(z,\lambda)$, in which 
\begin{align*}
    &m_3(z,\lambda)=1 + 2 m_4^{\kappa} \exp\left(\frac{A_2}{\log 2} \left(1 + A_1 \kappa +\frac{A_1A_2}{\log 2} \right) + \frac{L}{\log{2}} -\lambda + \left(\frac{2\kappa}{\lambda} + \frac{A_2}{\log{z}}\right)e^{\lambda}\right),\\
    &m_4 = 2\kappa e + \frac{A_2 e}{\log{2}} + \log{2}.
\end{align*}
\end{theorem}

\begin{proof}
Let
\begin{equation*}
    G(x,z) = \sum_{\substack{d<x\\d|P(z)}} g(d),
\end{equation*}
so that $G(z) = G(z,z)$. Now, \cite[(4.1.6)]{HalberstamRichert} informs us that for $\lambda > 0$,
\begin{equation}\label{eqn:originalplaya}
    \frac{1}{G(x,z)} \leq  W(z) \left(1 + \frac{\exp\left(-\lambda\frac{\log{x}}{\log{z}} + \left(\frac{2\kappa}{\lambda} + \frac{A_2}{\log{z}}\right)e^{\lambda}\right)}{G(x,z) W(z)}\right).
\end{equation}
Follow similar logic to \cite[p.~132--3]{HalberstamRichert} and apply Theorem \ref{theo:W} to see for $x\geq z$ that
\begin{align}
    \frac{1}{G(x,z) W(z)} 
    &\leq  \frac{W(z^{1/m_4})}{W(z) W(z^{1/m_4}) G(z,z^{1/m_4})} 
    \leq  2 \frac{W(z^{1/m_4})}{W(z)} \nonumber\\
    &\leq  2 m_2 \exp\left(\kappa\log\log{2} + \frac{L}{\log{2}}\right) m_4^{\kappa} \nonumber\\
    &= 2 \exp\left(\frac{A_2}{\log 2} \left( 1+A_1 \kappa +\frac{A_1A_2}{\log 2} \right) + \frac{L}{\log{2}}\right) m_4^{\kappa} . \label{eqn:GW1}
\end{align}
Finally, insert \eqref{eqn:GW1} with $x=z$ into \eqref{eqn:originalplaya} to obtain 
\begin{equation*}
    \frac{1}{G(z)} 
    = \frac{1}{G(z,z)} 
    \leq  W(z) \left(1 + \frac{\exp\left(-\lambda + \left(\frac{2\kappa}{\lambda} + \frac{A_2}{\log{z}}\right)e^{\lambda}\right)}{G(z,z) W(z)}\right) \leq  W(z) m_3(z,\lambda). \qedhere
\end{equation*}
\end{proof}

In the following result, we prove an explicit version of \cite[Lem.~5.5.4]{HalberstamRichert}, which tells us that when $z$ is large we can be more precise than Theorem \ref{theo:G}, which tells us $G^{-1}(z) \ll W(z)$. For completeness, we will incorporate the bound in Theorem \ref{theo:G}.

\begin{lemma}\label{lem:1/G}
Suppose that $z$ satisfies $r(z) < 1$, $\widehat{m_0}(z) < \log z$, and $\left|\frac{(\kappa +1) r(z)}{1-r(z)}\right| < 1$. Let
\begin{align}
    m_5(z,\lambda) &= \min\left\{\log{z} \left(\frac{m_3(z,\lambda)}{e^{\kappa\gamma}\Gamma (\kappa+1)}\left(1+\frac{\widehat{m_0}(z)}{\log z}\right)-1\right), m_6(z)\left( 1+\frac{m_6(z)}{\log{z}}\right)^{-1}\right\},\nonumber\\
    m_6(z) &= \frac{r(z)\log z}{1-r(z)}+m_7(z)+\frac{m_7(z)r(z)}{1-r(z)}, \nonumber \\
    m_7(z) &= \frac{\exp \left( \frac{(\kappa +1) r(z)}{(1-r(z))}\right)-1}{\log z}, \label{eq:m2} \\
    r(z) &= \frac{A_2+m_1(1,1/2)}{\log z}. \label{eq:r} 
\end{align}
The size of $m_6(z)$ is non-trivial and
\begin{equation*}
    \frac{1}{G(z)} \leq \frac{\Gamma(\kappa +1)}{\log^\kappa{z}} \left(1 + \frac{m_5(z,\lambda)}{\log{z}}\right) \prod_p\left( 1-\frac{\omega(p)}{p}\right)\left( 1-\frac{1}{p}\right)^{-\kappa} .
\end{equation*}
\end{lemma}

\begin{proof}
The left-hand term of $m_{5}$ follows directly from Theorem \ref{theo:G} and Lemma \ref{lem:W}. To prove the rest of the result, let $p$ be a prime divisor of $P(z)$ and follow the steps in \cite[p.~148-9]{HalberstamRichert}. Begin by noting that Halberstam and Richert showed
\begin{align*}
    \sum_{\substack{d<x \\ d|P(z)}} g(d)\log d
    &=\sum_{\substack{d<x\\ d|P(z)}} g(d) \sum_{p<\min(x/d,z)}\frac{\omega(p)}{p}\log p \\
    &+\sum_{\substack{x z^{-2}<d<x\\ d|P(z)}} g(d)\sum_{\substack{\sqrt{x/d}\leq p< \min(x/d,z)\\ p\nmid d}}\frac{g(p)\omega(p)}{p}\log p. 
\end{align*}
Apply \eqref{eq:omega2} and \eqref{eq:3}, to obtain
\begin{align*}
    \sum_{\substack{d<x \\ d|P(z)}}g(d)\log d
    &=\sum_{\substack{x/z\leq d<x\\ d|P(z)}} g(d) \left(\kappa\log \frac{x}{d}+O^*(A_2) \right) \\
    &\qquad\qquad+\sum_{\substack{d<x/z\\ d|P(z)}} g(d)\left(\kappa\log z +O^*(A_2) \right) + O^*(m_1(x, x/2) G(x,z))\\
    &= \kappa \sum_{\substack{d<x\\ d|P(z)}} g(d)\log \frac{x}{d} - \kappa \sum_{\substack{d<x/z\\ d|P(z)}}\log \frac{x/z}{d}+O^*((A_2+m_1(x,x/2))G(x,z)).
\end{align*}
Here, it is of interest to note that we took the worst value for $m_1(x,d)$, but it should be possible to slightly improve the result by being more careful. Next, add the first sum on the right to both sides, and introduce the function
\begin{align*}
    T(x,z):=\int_1^x G(t,z)\frac{dt}{t}=\sum_{\substack{d<x\\ d|P(z)}} g(d)\log \frac{x}{d}.
\end{align*}
Doing this, we obtain 
\begin{align*}
    G(x,z)\log x = (\kappa +1)T(x,z) - \kappa T\left( \frac{x}{z},z\right)+O^*((A_2+m_1(x,x/2))G(x,z)).
\end{align*}
In particular, let $x=z$ and recall $G(x,z)=G(x)$ for $z\geq x$ from \cite[(1.4.23)]{HalberstamRichert}. It follows that
\begin{align*}
    T(z):=T(z,z)=\int_1^z G(t)\frac{dt}{t}.
\end{align*}
Therefore, we have
\begin{align*}
    G(z)\log z = (\kappa +1) T(z) + G(z) O^*(r(z)) \log z;
\end{align*}
recall that $r(z)$ was defined in \eqref{eq:r}. Now, for $y\geq  z$ write
\begin{equation}\label{eq:Gz}
    G(z) \geq \frac{T(z)}{1-r(z)}\frac{\kappa + 1}{\log z}
    \quad\text{and}\quad
    E(y) = \log\left(\frac{\kappa + 1}{\log^{\kappa+1}{y}} T(y)\right).
\end{equation}
Here, it is now clear why the assumption $r(z)<1$ is important. Moreover, if $y\geq z$, then we have by \eqref{eq:Gz} that the integral $\int_z^{\infty}E'(y)\,dy$ converges, because
\begin{align*}
    E'(y)
    &=\frac{\kappa +1}{y \log y}\frac{r(y)}{1-r(y)}
    =O\left(\frac{1}{y\log^2 y}\right);
\end{align*}
see \cite[p.~150]{HalberstamRichert} for extra details. Therefore, there exists a constant $c_3$ such that
\begin{align*}
    \frac{\kappa +1}{\log^{\kappa +1}{z}} T(z) 
    = \exp(E(z))
    = c_3 \exp\left(-\int_z^{\infty} E'(y)dy\right)
    \geq  c_3 \left( 1+ \frac{m_7(z)}{\log z}\right),
\end{align*}
where $m_7$ as defined in \eqref{eq:m2}; this aspect of the proof clarifies why the remaining assumption is important. One can rearrange this identity to see
\begin{equation}\label{eq:Gz1}
    T(z)\geq \frac{c_3}{\kappa +1}\log^{\kappa +1}{z} \left(1+ \frac{m_7(z)}{\log z}\right),
\end{equation}
but $\frac{1}{1-r(z)}=1+\frac{r(z)}{1-r(z)}$, so that we conclude from \eqref{eq:Gz} and \eqref{eq:Gz1} that
\begin{align*}
    G(z)\geq c_3 \log^{\kappa}{z} \left( 1+\frac{1}{\log{z}}\left(\frac{r(z)\log z}{1-r(z)}+m_7(z)+\frac{m_7(z)r(z)}{1-r(z)}\right)\right).
\end{align*}
Finally, import the following observation from \cite[(5.3.13)]{HalberstamRichert} to complete the result: 
\begin{align*}
    c_3 &= \frac{1}{\Gamma(\kappa +1)}\prod_p \left( 1-\frac{\omega (p)}{p}\right)^{-1} \left( 1-\frac{1}{p}\right)^{\kappa} . \qedhere 
\end{align*}
\end{proof}

We can now prove the main result. 

\begin{theorem}\label{theo:sieve}
Assume \eqref{eq:R}, \eqref{eq:omega1}, \eqref{eq:omega2}, $z_0^2 = X/\log^{4\kappa + 1}{X}$, and $X$ satisfies 
\begin{equation}\label{eqn:conditions}
    z_0 > \max\{2, e^{A_1A_2}\}, \quad
    r(z_0) < 1, \quad 
    \widehat{m_0}(z_0) < \log{z_0}, \quad
    \left|\frac{(\kappa +1) r(z_0)}{1-r(z_0)}\right| < 1.
\end{equation}
We have that $S(\cAA;\cPP,\sqrt{X})$ is majorised by 
\begin{equation}\label{eq:fin1}
    \frac{2^{\kappa} \Gamma(\kappa +1)  X}{\log^{\kappa}{X}} \left( 1+\frac{m_8(z_0,\lambda)}{\log X}\right)\left(1 + \frac{m_9 \log\log{X}}{\log{X}}\right)^{\kappa} \prod_p\left( 1-\frac{\omega(p)}{p}\right)\left( 1-\frac{1}{p}\right)^{-\kappa},
\end{equation}
where the infinite product on the right is convergent,
\begin{align*}
    m_8(z_0,\lambda) = \frac{2m_5(z_0,\lambda)}{1-(4\kappa+1)\frac{\log \log X}{\log X} } + \frac{2^{-4\kappa}  m_2^4}{\Gamma(\kappa +1) e^{\kappa\gamma}}\left(1 + \frac{\widehat{m_0}(z_0)}{\log z_0} \right), \quad
    m_{9} = \frac{4\kappa + 1}{1 - \frac{(4\kappa + 1)\log\log{X}}{\log{X}}} .
\end{align*}
\end{theorem}

\begin{proof}
Note that Lemma \ref{lem:W} and Theorem \ref{theo:W} imply
\begin{align*}
    \frac{z_0^2}{W^3(z_0)}
    = \frac{W(z_0) z_0^2}{W^4(z_0)}
    &\leq W(z_0) m_2^4 z_0^2 \log^{4\kappa}{z_0} \\
    &\leq m_2^4 e^{-\kappa\gamma} z_0^2 \log^{3\kappa}{z_0}\left(1 + \frac{\widehat{m_0}(z_0)}{\log{z_0}} \right) \prod_p\left( 1-\frac{\omega(p)}{p}\right)\left( 1-\frac{1}{p}\right)^{-\kappa}.
\end{align*}
Hence, apply Theorem \ref{thm:SelbergUpperSieve} with the above observation, the definition of $z_0^2$, and Lemma \ref{lem:1/G} to see that $S(\cAA,\cPP,z_0)$ is majorised by
\begin{align*}
    \frac{\Gamma(\kappa +1) X}{\log^{\kappa}{z_0}} \left(1 + \frac{m_5(z_0,\lambda)}{\log{z_0}} + \frac{m_2^4 \log^{4\kappa}{z_0}}{e^{\kappa\gamma} \log^{4\kappa + 1}{X}} \left(1 + \frac{\widehat{m_0}(z_0)}{\log{z_0}} \right) \right) \prod_p\left( 1-\frac{\omega(p)}{p}\right)\left( 1-\frac{1}{p}\right)^{-\kappa}.
\end{align*}
Next, note that $\frac{\log z}{\log z_0} = \left( 1+ \frac{\log z/z_0}{\log z_0}\right)$, so if $z = \sqrt{X}$, then
\begin{equation*}
    \frac{\log z}{\log z_0} 
    = 1+ \frac{m_9 \log\log{X}}{\log{X}} .
\end{equation*}
Finally, note that if $z_0\leq z \leq \sqrt{X}$, then $S(\cAA;\cPP,z)\leq S(\cAA;\cPP,z_0)$, so we can combine all of the previous observations to obtain the result after some straightforward manipulations.
\end{proof}

\section{Primes representable by a set of polynomials}
\label{S:app}


Suppose $F\in\Z[X]$ has degree $\deg_F \geq 1$ and $x\geq 1$. Let $\rho_F(p)$ be the number of solutions to $F(n) \equiv 0 \mod{p}$, assuming that $\rho_F(p) < p$ throughout to ensure that $F$ has no fixed prime divisor. To prove our results, we will apply Theorem \ref{thm:SelbergUpperSieve} with
\begin{equation*}
    \cAA = \{F(n) : x-y < n \leq  x\},\quad
    \cPP = \{p : p\text{ prime}\},\quad
    X = y,
    \quad\text{and}\quad
    \omega(p) = \rho_F(p).
\end{equation*}

\subsection{Properties}

We know the following properties of $\rho_F$:
\begin{enumerate}
    \item $\rho_F$ is multiplicative. 
    \item Lagrange's theorem for congruences ensures $\rho_{F}(p) \leq \deg_F$.
    \item For square-free $d$, we have
    \begin{equation*}
        \#\cAA_d = \rho_F(d)\left(\frac{y}{d} + \vartheta\right), \qquad |\vartheta|\leq  1.
    \end{equation*}
\end{enumerate}
Immediately, condition \eqref{eq:R} holds under these choices, and condition \eqref{eq:omega1} holds with $A_1 = \deg_F +1$. It is more complicated to find explicit values for $L$ and $A_2$ which satisfy \eqref{eq:omega2}; for this task we need to import the following auxiliary result.

\begin{theorem}\label{thm:NagellTypeSums}
Let $F\in\Z[X]$ be irreducible with degree $d\geq 1$, leading coefficient $c$, discriminant $D_F$, and weighted discriminant $\D_F = |c|^{(d-1)(d-2)} |D_F|$. For $x\geq \max\{2,\sqrt{\D_F}\}$, there exists $Q_F = O(1)$ such that
\begin{align*}
    \left|\sum_{p\leq  x} \frac{\rho_F(p) \log{p}}{p} - \log{x}\right|
    \leq Q_F .
\end{align*}
In particular, suppose that $\overline{M}_{\Q}(x) = \sum_{p\leq  x}\frac{\log{p}}{p}$, $\mathfrak{m}(d) = (\pi/4)^{d} d^{2 d}/(d!)^2$, 
\begin{align}
    \lambda_{\K}(d) &=
    \begin{cases}
        (d+1)^{\frac{1}{2} - \frac{1}{2 d}} \left(\frac{5}{8} + \frac{\pi}{2} - \frac{1}{d} + \frac{3}{8 d^2}\right)^{\frac{1}{2}} e^{d\left(2.27 + \frac{4 d}{d-1} + \frac{0.01}{d^2} + \frac{1}{500 d^6}\right)} &\text{if }d \leq 13, \\
        (d+1)^{d-\frac{1}{2}-\frac{1}{2\,d}} \left(\frac{5}{8} + \frac{\pi}{2} + \frac{1}{d} + \frac{3}{8 d^2}\right)^{\frac{1}{2}} e^{4.13 d + \frac{0.02}{d}} &\text{if }d > 13,
    \end{cases} \nonumber\\
    \Lambda_F(d) &= \begin{cases}
        0 & \text{if }d=1,\\
        \frac{0.54 (3 d-1) \lambda_{\K}(d)}{(d-1)^2 (\log{\mathfrak{m}(d)})^{d-1}} d^{3/2} d! \,\D_F^{\frac{1}{d+1}} (\log{\D_F})^{d-1} & \text{if }d\geq 2,
    \end{cases}\nonumber\\
    C_F(d) &= 1.38 \frac{(d+1)^2}{(d-1)} + 1.52\, d(d+1) + 111.26\, d. \nonumber
\end{align}
Unconditionally, $Q_F = d \left(\overline{M}_{\Q}(|c|) + \overline{M}_{\Q}(\sqrt{\D_F}) + 2.52\right) + 1 + \Lambda_F(d) C_F(d) \sqrt{\D_F}$ is admissible. However, if the Generalised Riemann Hypothesis (GRH) for Dedekind zeta-functions is true, then $Q_F$ may be refined into
\begin{align*}
    Q_F = d \left(\overline{M}_{\Q}(|c|) + \overline{M}_{\Q}(\sqrt{\D_F}) + 10.79\right) + \log{2} + 4.73 \log{\D_F}.
\end{align*}
\end{theorem}

\begin{proof}
The second author proves this in his upcoming PhD thesis \cite{LeeThesis}, by using and extending methods from \cite{GarciaLeeSuhYu}. That is, there is an explicit connection between $\rho_F(p)$ and the number of prime ideals with norm $p$ in a certain number field, so the sum in question is explicitly related to one of Mertens' theorems for number fields which was established with explicit constants in \cite{EMT4NF2}; the condition $x\geq \max\{2,\sqrt{\D_F}\}$ is necessary so that one can apply this relationship. Since \cite{EMT4NF2} was published, the second author has updated a key ingredient in the proof of the required Mertens' theorem in \cite{LeeICF} and used that to refine the required Mertens' theorem in \cite{LeeThesis}; this is the origin of the complicated looking functions $\lambda_{\K}(d)$ and $\Lambda_F(d)$.
\end{proof}

We will also need the following preliminary result, which is imported from \cite[Lem.~11]{GarciaLeeSuhYu}; it is a useful additive splitting property for $\rho_f(p)$ when $f\in\Z[X]$ is reducible.

\begin{lemma}\label{Lemma:OmegaSplit}
Suppose that $F =F_1\cdots F_g \in \Z[x]$ is a product of distinct, irreducible, non-constant polynomials $F_i \in \Z[x]$. If $p > |D_F|$, then $\rho_F(p) = \rho_{F_1}(p) + \cdots + \rho_{F_g}(p)$.
\end{lemma}

Using Theorem \ref{thm:NagellTypeSums} and Lemma \ref{Lemma:OmegaSplit}, we can establish the following result in Corollary \ref{cor:LF}, which shows that if $\kappa = g$, then \eqref{eq:omega2} is true with 
\begin{equation}\label{eq:A2L}
    L = A_2 = \L_F := 2
    \begin{cases}
        \max\{g,\deg_F -1\} \log{\max\{2, \sqrt{\D_F}\}} + g \max_i Q_{F_i} &\text{if } g\geq 2 , \\
        \max\left\{\max\{1,\deg_F -1\} \log{\max\{2, \sqrt{\D_F}\}}, Q_F\right\} &\text{if } g = 1 .
    \end{cases}
\end{equation}

\begin{corollary}\label{cor:LF}
Let $F_1, \dots, F_g\in\Z[X]$ be distinct irreducible polynomials with positive leading coefficients and $F = F_1\cdots F_g$. For any $1<w<z$, $L_F = \L_F$ is admissible in the following:
\begin{equation}\label{eqn:LF}
    \left|\sum_{w \leq p < z} \frac{\rho_F(p) \log{p}}{p} - g \log{\frac{z}{w}}\right| \leq L_F .
\end{equation}
\end{corollary}

\begin{proof}
We will need the following important observation, which follows from \cite[(3.24)]{Rosser} and $0\leq\rho_F(p) \leq \deg_F$:
\begin{equation}\label{eqn:smallx}
    \left|\sum_{p < x} \frac{\rho_F(p) \log{p}}{p} - g \log{x}\right| \leq \max\{g,\deg_F -1\} \log{x}.
\end{equation}
First, assume that $w \geq |D_F|$, then Theorem \ref{thm:NagellTypeSums} and Lemma \ref{Lemma:OmegaSplit} imply
\begin{equation*}
    \sum_{w \leq p < z} \frac{\rho_F(p) \log{p}}{p}
    = \sum_{i\leq g} \sum_{w \leq p < z} \frac{\rho_{F_i}(p) \log{p}}{p}
    = g \log{\frac{z}{w}} + O^*\!\left(2 \sum_{i\leq g} Q_{F_i}\right) .
\end{equation*}
If $w > \max\{2, \sqrt{\D_F}\}$ too, then it follows from Theorem \ref{thm:NagellTypeSums} that \eqref{eqn:LF} holds with
\begin{equation*}
    L_F = 2 g \max_{i} Q_{F_i},
\end{equation*}
Otherwise, $w < |D_F|$ and it follows from Lemma \ref{Lemma:OmegaSplit} and $\D_{F_i} \leq \D_F$ that
\begin{align*}
    \sum_{w \leq p < z} \frac{\rho_F(p) \log{p}}{p}
    = \sum_{w \leq p < \min\left\{M_F, z\right\}} \frac{\rho_{F}(p) \log{p}}{p} + \chi_F \sum_{i\leq g} \sum_{M_F \leq p < z} \frac{\rho_{F_i}(p) \log{p}}{p},
\end{align*}
in which $M_{F} = \max\left\{2, \sqrt{\D_{F}}\right\}$ and
\begin{equation*}
    \chi_F = 
    \begin{cases}
        0 &\text{if } z \leq M_F , \\
        1 &\text{if } z > M_F .
    \end{cases}
\end{equation*}
If $\min\left\{M_F, z\right\} = z$, then \eqref{eqn:smallx}, Theorem \ref{thm:NagellTypeSums}, and Lemma \ref{Lemma:OmegaSplit} imply that \eqref{eqn:LF} holds with
\begin{equation*}
    L_F = 2 \max\{g,\deg_F -1\} \log{M_F} .
\end{equation*}
If $\min\left\{M_F, z\right\} = M_F$, then \eqref{eqn:smallx}, Theorem \ref{thm:NagellTypeSums}, and Lemma \ref{Lemma:OmegaSplit} imply that \eqref{eqn:LF} holds with
\begin{equation*}
    L_F = 2 \left(\max\{g,\deg_F -1\} \log{M_F} + g \max_{i} Q_{F_i}\right) .
\end{equation*}
Finally, if $g=1$, then $F$ is irreducible and we can do better. That is, \eqref{eqn:LF} holds with
\begin{equation*}
    L_F = 2 \max\left\{\max\{1,\deg_F -1\} \log{M_F}, Q_F\right\} . \qedhere
\end{equation*} 
\end{proof}

\subsection{Selberg's upper sieve for primes representable by a polynomial}

We can now prove an explicit version of \eqref{eq:RH} (or \cite[Thm.~5.5.3]{HalberstamRichert}).

\begin{theorem}\label{theo:general}
Let $F = F_1\cdots F_g$ such that $F_1, \dots, F_g\in\Z[Y]$ are distinct, irreducible polynomials with positive leading coefficients and $\rho_F(p) < p$ for all primes $p$. Let $x \geq y$ be real numbers, $\lambda > 0$, and assert $X=y$, $\omega(p) = \rho_F(p)$, $A_1=\deg_F +1$, $A_2=L=\L_{F}$, and $\kappa=g$. If $y$ satisfies the conditions in \eqref{eqn:conditions} and $y_0 = \sqrt{y} / (\log{y})^{2g + 1/2}$, then
\begin{equation*}
    \begin{split}
        \mathfrak{F}_F(x,y)
        \leq \frac{2^{g} \Gamma(g +1)  y}{\log^{g}{y}} \left( 1+\frac{m_8(y_0,\lambda)}{\log{y}}\right)&\left(1 + \frac{m_9 \log\log{y}}{\log{y}}\right)^{g} \prod_p\left( 1-\frac{\rho_F(p)}{p}\right)\left( 1-\frac{1}{p}\right)^{-g} \\
        &+ \max\{n:\text{there exists}~ i ~\text{such that} ~ F_i(n)<\sqrt{y}\} ,
    \end{split}
\end{equation*}
in which $\mathfrak{F}_F(x,y) = \#\{x-y<n\leq x: F_i(n)~\text{are simultaneously prime for}~1\leq i\leq g\}$.
\end{theorem}

\begin{proof}
Assume the choices in the statement of the result and apply these in Theorem \ref{theo:sieve} to see that $S(\cAA;\cPP,\sqrt{y})$ is majorised by
\begin{equation*}
    \frac{2^{g} \Gamma(g+1) y}{\log^{g}{y}} \left( 1 + \frac{m_8(y_0,\lambda)}{\log{y}}\right) \left(1 + \frac{m_9 \log\log{y}}{\log{y}}\right)^{g} \prod_p\left( 1-\frac{\rho_F(p)}{p}\right) \left( 1-\frac{1}{p}\right)^{-g} .
\end{equation*}
Use \eqref{eq:omega1} and Lagrange's theorem to justify the choice $A_1=\deg_F +1$ and
\eqref{eq:A2L} to justify choosing $L=A_2=\L_F$. Next, for each $n$ that is in $\mathfrak{F}_F(x,y) $ but is not counted by $S(\cAA;\cPP,\sqrt{y})$, we have $F_i(n)<\sqrt{y}$ for at least one $i$, therefore
\begin{equation}\label{eq:final}
    \mathfrak{F}_F(x,y) \leq S(\cAA;\cPP,\sqrt{y})+\max\{n:\text{there exists}~ i ~\text{such that} ~ F_i(n)<\sqrt{y}\}.
\end{equation}
The result follows by applying Theorem \ref{theo:sieve} to \eqref{eq:final}.
\end{proof}

Next, we prove an explicit version of \cite[Thm.~5.5]{HalberstamRichert}, which extends Theorem \ref{theo:general}.

\begin{theorem}\label{theo:generalp}
Let $F = F_1\cdots F_g$ such that $F'(k) = k F(k)$, $F_i(k) \neq k$, and $F_i\in\Z[X]$ are distinct irreducible polynomials with positive leading coefficients. Moreover, let $\rho_F(p) < p -1$ if $p \nmid F(0)$ and
\begin{equation*}
    \rho'_{F} (p) = 
    \begin{cases}
        \rho_{F} (p)+1 & \text{if } p \nmid F(0), \\ 
        \rho_{F} (p) & \text{if } p ~|\, F(0).
    \end{cases}
\end{equation*}
Finally, let $x \geq y$ be real numbers, $\lambda > 0$, and assert $X=y$, $\omega(p) = \rho'_F(p)$, $A_1=\deg_F +2$, $A_2=L=\L_{F'}$, and $\kappa=g+1$. If $y$ satisfies the conditions in \eqref{eqn:conditions}, $y'_0 = \sqrt{y} / (\log{y})^{2g + 5/2}$, and
\begin{equation*}
    \mathfrak{G}_F(x,y) = \#\{x-y<p\leq x : p~\text{prime and}~ F_i(p)~\text{are simultaneously prime for}~1\leq i\leq g\},
\end{equation*}
then $\mathfrak{G}_F(x,y)$ is majorised by
\begin{equation*}
    \begin{split}
        \frac{2^{g+1} \Gamma(g +2) y}{\log^{g+1}{y}} &\left( 1+\frac{m_8(y'_0,\lambda)}{\log{y}}\right) \left(1 + \frac{m_9 \log\log{y}}{\log{y}}\right)^{g+1} \prod_p\left( 1-\frac{\rho'_F(p)}{p}\right)\left( 1-\frac{1}{p}\right)^{-g-1} \\
        &\quad\qquad + \max\{n:\text{there exists}~ i ~\text{such that} ~ F_i(n)<\sqrt{y}\} .
    \end{split}
\end{equation*}
\end{theorem}

\begin{proof}
Clearly, $\rho'_{F}(k)$ is the number of solutions of $F'(k) \equiv 0 \pmod{p}$ and the condition $\rho_{F}(p) < p-1$ implies $\rho'_{F} (p)<p$. So, apply the choices in the statement of the result into Theorem \ref{theo:sieve} to see that $S(\cAA;\cPP,\sqrt{y})$ is majorised by
\begin{equation*}
    \frac{2^{g+1} \Gamma(g+2) y}{\log^{g+1}{y}} \left( 1 + \frac{m_8(y'_0,\lambda)}{\log{y}}\right) \left(1 + \frac{m_9 \log\log{y}}{\log{y}}\right)^{g+1} \prod_p\left( 1-\frac{\rho'_F(p)}{p}\right) \left( 1-\frac{1}{p}\right)^{-g-1} .
\end{equation*}
The result follows by noting
\begin{equation*}
    \mathfrak{G}_F(x,y) \leq S(\cAA;\cPP,\sqrt{y}) + \max\{n:\text{there exists}~ i ~\text{such that} ~ F_i(n)<\sqrt{y}\}. \qedhere
\end{equation*}
\end{proof}


\begin{remark}
If $F_i(x) = \sum_{j=0}^{\deg_{F_i}} a_{i,j} x^{j}$, then we can easily obtain
\begin{equation*}
    \max\{n:\text{exists}~ i ~\text{such that} ~ F_i(n)<\sqrt{y}\}\le \max_i\left\{  y^{\frac{1}{2 (\deg_{F_i}-1)}},  \sum_{j=0}^{\deg_{F_i}-1} \frac{|a_{i,j}|}{a_{i,\deg_{F_i}}}\right\}.
\end{equation*}
To make \eqref{eq:final} more compatible with \eqref{eq:RH}, we also recall \cite[(5.2.6)]{HalberstamRichert}, which gives
\begin{equation}\label{eqn:prod_lower}
    \prod_{p}\left(1-\frac{\omega(p)}{p}\right)\left(1-\frac{1}{p}\right)^{-\kappa}
    \geq \exp\left(-A_1A_2(1+\kappa+A_2)\right) > 0.
\end{equation}
\end{remark}

\section{Results}\label{S:Results}

In the remainder of this paper, we apply Theorem \ref{theo:general} to some choices of $F\in\Z[X]$ and deduce the main results of this paper. In particular, we prove Corollaries \ref{thm:Main} and \ref{thm:Main1}. To make our computations simpler, we will choose $\lambda=2g$ throughout, although this choice does \textit{not} have a great effect, because we expect the minimum function in the definition of $m_5(z,\lambda)$ to evaluate as the term not depending on $\lambda$. 

\subsection{Proof of Corollary \ref{thm:Main}}

Suppose $F \in\Z[X]$ is irreducible with degree $\deg_{F} \geq 1$, leading coefficient $c \geq 1$, and discriminant $D_F$. Note that $\pi_{F}(x) = \mathfrak{F}_{F}(x,x)$, so we can apply Theorem \ref{theo:general} with $x=y$ for sufficiently large $x$ to see $\pi_{F}(x)$ is majorised by
\begin{align*}
    \frac{2 x}{\log{x}} \left(1 + \frac{c_0(x) \log\log{x}}{\log{x}}\right) \prod_p\left( 1-\frac{\rho_{F}(p)}{p}\right)\left( 1-\frac{1}{p}\right)^{-1} \!\!\!+ \underbrace{\max\left\{  x^{\frac{1}{2 (\deg_{F}-1)}},  \sum_{j=0}^{\deg_{F}-1} \frac{|a_{j}|}{a_{\deg_{F}}}\right\}}_{\mathfrak{m}_F(x)},
\end{align*}
in which $a_j$ are the $j$th coefficients of $F$, $x_0 = \sqrt{x} / (\log{x})^{5/2}$, and
\begin{equation*}
    c_0(x) := m_{10}(x_0,2) = m_9 + \frac{m_8(x_0,2)}{\log\log{x}} + \frac{m_8(x_0,2) m_9}{\log{x}} .
\end{equation*}
There exist minimal choices of $b_0 \in (0,10)$ and $b_1 \in\mathbb{N}_{>0}$ such that the conditions in \eqref{eqn:conditions} are satisfied for all $x \geq X := \exp(b_0\cdot 10^{b_1})$, which we can find using computations. It also follows from \eqref{eqn:prod_lower} that for all $x \geq X$,
\begin{align*}
    \mathfrak{m}_F(x)
    &\leq c_1(X) \left( 1+\frac{c_0(X) \log\log{x}}{\log{x}}\right)  \prod_p\left( 1-\frac{\rho_F(p)}{p}\right)\left( 1-\frac{1}{p}\right)^{-1} \frac{x}{\log{x}},
\end{align*}
where $c_1(x) = \exp\left(2 \L_{F} (2 + \L_{F})\right) \mathfrak{m}_F(x) x^{-1} \log{x}$.  It follows that
\begin{align*}
    \pi_{F}(x)
    &< 2\left(1 + \frac{c_1(X)}{2}\right) \left(1 + \frac{c_0(X) \log\log{x}}{\log{x}}\right) \prod_p\left( 1-\frac{\rho_{F}(p)}{p}\right)\left( 1-\frac{1}{p}\right)^{-1} \frac{x}{\log{x}} \\
    &< 2 \left(1 + \frac{c_2(X) \log\log{x}}{\log{x}}\right) \prod_p\left( 1-\frac{\rho_{F}(p)}{p}\right)\left( 1-\frac{1}{p}\right)^{-1} \frac{x}{\log{x}},
\end{align*}
in which
\begin{equation*}
    c_2(X) = c_0(X)\left(1 + \frac{ c_1(X)}{2}\right) + \frac{c_1(X) \log{X}}{2 \log\log{X}} .
\end{equation*}

To obtain the results presented in Corollary \ref{thm:Main}, we make explicit choices of irreducible polynomial, then use computations to determine the minimal choices of $b_0$ and $b_1$ such that the conditions in \eqref{eqn:conditions} hold and values for $c_2(X)$ under these minimal choices. First, choose $F_0(k) = k^2 + 3$, which satisfies $\deg_{F_0} = 2$, $c = 1$, and $|D_{F_0}| = \D_{F_0} = 12$. In this case, computations revealed that the result is complete for all $x\geq X_0$ and $\tau_0 = c_2(X_0)$, where
\begin{equation*}
    X_0 =
    \begin{cases}
        e^{1.5\cdot 10^{38}} &\text{unconditionally,} \\
        e^{5.5\cdot 10^{7}} &\text{if the GRH is true,}
    \end{cases}
    \qquad
    c_2(X_0) =
    \begin{cases}
        e^{2.40829\cdot 10^{25}} &\text{unconditionally,} \\
        e^{1.28266\cdot 10^{5}} &\text{if the GRH is true.}
    \end{cases}
\end{equation*}
Second, we choose $F_1(k) = k^3 - 5$, which satisfies $\deg_{F_1} = 3$, $c = 1$, and $|D_{F_1}| = \D_{F_1} = 675$. In this case, computations revealed that the result is complete for all $x\geq X_1$ and $\tau_1 = c_2(X_1)$, where
\begin{equation*}
    X_1 =
    \begin{cases}
        e^{1.8\cdot 10^{50}} &\text{unconditionally,} \\
        e^{5.7\cdot 10^{8}} &\text{if the GRH is true,}
    \end{cases}
    \qquad
    c_2(X_1) =
    \begin{cases}
        e^{3.00518\cdot 10^{33}} &\text{unconditionally,} \\
        e^{6.72301\cdot 10^{5}} &\text{if the GRH is true.}
    \end{cases}
\end{equation*} 
Third, choose $F_2(k) = k^5 + 3$, which satisfies $\deg_{F_2} = 5$, $c = 1$, and $|D_{F_2}| = \D_{F_2} = 253\,125$. In this case, computations revealed that the result is complete for all $x\geq X_2$ and $\tau_2 = c_2(X_2)$, where
\begin{equation*}
    X_2 =
    \begin{cases}
        e^{6.1\cdot 10^{75}} &\text{unconditionally,} \\
        e^{6.5\cdot 10^{9}} &\text{if the GRH is true,}
    \end{cases}
    \qquad
    c_2(X_2) =
    \begin{cases}
        e^{3.69160\cdot 10^{50}} &\text{unconditionally,} \\
        e^{3.87306\cdot 10^{6}} &\text{if the GRH is true.}
    \end{cases}
\end{equation*}  
Finally, choose $F_3(k) = 2k^6 + 3$, which satisfies $\deg_{F_3} = 5$, $c = 1$, $|D_{F_3}| = 362\,797\,056$, and $\D_{F_3} = 380\,420\,285\,792\,256$. In this case, computations revealed that the result is complete for all $x\geq X_3$ and $\tau_3 = c_2(X_3)$, where
\begin{equation*}
    X_3 =
    \begin{cases}
        e^{2.8\cdot 10^{109}} &\text{unconditionally,} \\
        e^{9.3\cdot 10^{10}} &\text{if the GRH is true,}
    \end{cases}
    \qquad
    c_2(X_3) =
    \begin{cases}
        e^{1.06883\cdot 10^{73}} &\text{unconditionally,} \\
        e^{2.40569\cdot 10^{7}} &\text{if the GRH is true.}
    \end{cases}
\end{equation*}  

\subsection{Proof of Corollary \ref{thm:Main1}}

To prove Corollary \ref{thm:Main1}, we will apply Theorem \ref{theo:generalp} with $F(k) = 2k+1$ (so that $F'(k) = k(2k+1)$). Note that $\rho_F(2) = 0$ and $\rho_F(p) = 1$ for all primes $p > 2$, because there is only one solution modulo $p$ at $k = (p-1)/2$ for $p>2$; this means that the condition $\rho_F(p) < p-1$ is certainly satisfied. Moreover, it follows that $\rho'_F(2) = 1$ and $\rho'_F(p) = 2$ for all primes $p > 2$, so
\begin{equation}\label{eqn:observation_fun}
    \prod_p\left( 1-\frac{\rho'_F(p)}{p}\right)\left( 1-\frac{1}{p}\right)^{-2}
    = 2 \prod_{p>2}\left( 1-\frac{2}{p}\right)\left( 1-\frac{1}{p}\right)^{-2} 
    = 2 \prod_{p>2} \left( 1-\frac{1}{(p-1)^2}\right) . 
\end{equation}
Next, recall that $\mathfrak{G}_F(x,x) = \pi_{F'}(x) = \#\{x-y<p \leq x : p~\text{and}~ 2p+1 ~\text{are prime}\}$. Therefore, if $x$ is large enough, then it follows from Theorem \ref{theo:generalp} and \eqref{eqn:observation_fun} that 
\begin{equation*}
    \pi_{F'}(x)
    \leq 16 \left( 1+\frac{m_8(x'_0,\lambda)}{\log{x}}\right)\left(1 + \frac{m_9 \log\log{x}}{\log{x}}\right)^2  \prod_{p>2} \left( 1-\frac{1}{(p-1)^2}\right) \frac{x}{\log^2{x}} + \sqrt{\frac{x-1}{2}} ,
\end{equation*}
in which $x'_0 = \sqrt{x} / (\log{x})^{9/2}$. We know that $\deg_F = 1$, $c = 2$, and $|D_F| = \D_F = 1$, so we compute that $\log{x} \geq 1.3\cdot 10^{6}$ is large enough. Moreover,
\begin{align*}
    \left( 1+\frac{m_8(x'_0,\lambda)}{\log{x}}\right)\left(1 + \frac{m_9 \log\log{x}}{\log{x}}\right)^2
    &= 1 + \frac{m_{11}(x'_0,\lambda) \log\log{x}}{\log{x}} ,
\end{align*}
in which
\begin{align*}
    m_{11}(y,\lambda) = 2m_9+ \frac{m_9^2\log \log y}{\log y}+\frac{m_8(x'_0,\lambda)}{\log \log y}\left( 1+\frac{2m_9\log \log y}{\log y}+\left(\frac{m_9\log \log y}{\log y}\right)^2\right).
\end{align*}
Recall that Wrench computed the last product in \eqref{eqn:observation_fun} to $42$ decimal places in \cite{Wrench}; it is larger than $0.6601618158$. Therefore, for all $\log{x} \geq 1.3\cdot 10^{6}$, $\log{m_{11}(x'_0,\lambda)} \leq 9.03885\cdot 10^{3}$ and 
\begin{align*}
    \sqrt{\frac{x-1}{2}} 
    < \sqrt{x} 
    &\leq \left(\frac{(1.3\cdot 10^{6})^2}{\sqrt{e^{1.3\cdot 10^{6}}}}\right)\frac{x}{\log^2{x}} \\
    &\leq 2 \left(\frac{(1.3\cdot 10^{6})^2}{2\cdot 0.6601618158 \sqrt{e^{1.3\cdot 10^{6}}}}\right) \prod_{p>2} \left( 1-\frac{1}{(p-1)^2}\right) \frac{x}{\log^2{x}} \\
    &\leq 16 \left(\frac{(1.3\cdot 10^{6})^2}{16\cdot 0.6601618158\sqrt{e^{1.3\cdot 10^{6}}}}\right) \prod_{p>2} \left( 1-\frac{1}{(p-1)^2}\right) \frac{x}{\log^2{x}} \\
    &< 16 e^{-6.49974\cdot 10^{5}} \prod_{p>2} \left( 1-\frac{1}{(p-1)^2}\right) \frac{x}{\log^2{x}} .
\end{align*}
It follows that for all $\log{x} \geq 1.3\cdot 10^{6}$, we have
\begin{equation*}
    \pi_{F'}(x)
    \leq 16 \left(1 + e^{-6.49974\cdot 10^{5}}\right) \left(1 + \frac{e^{9.03885\cdot 10^{3}} \log\log{x}}{\log{x}}\right) \prod_{p>2} \left( 1-\frac{1}{(p-1)^2}\right) \frac{x}{\log^2{x}} .
\end{equation*}

\bibliographystyle{amsplain}
\bibliography{refs}

\end{document}